\documentclass[12pt,a4paper]{amsart}
\numberwithin{equation}{section}
\usepackage{graphicx}
\usepackage{latexsym}
\usepackage{amsxtra}
\usepackage{amsmath}
\usepackage{amsfonts}
\usepackage{amssymb}

\newtheorem{theorem}{Theorem}[section]
\newtheorem{corollary}{Corollary}[section]

\newtheorem{proposition}{Proposition}[section]

\usepackage[figuresright]{rotating}

\title{A family of sequences of binomial type}
\author{Wojciech M{\l}otkowski, Anna Romanowicz}
\thanks{
W.~M. is supported by the Polish
National Science Center grant No. 2012/05/B/ST1/00626,
A.~R. is supported by the Grant DEC-2011/02/A/ST1/00119
of National Centre of Science.}
\address{Instytut Matematyczny,
Uniwersytet Wroc{\l}awski,
Plac~Grunwaldzki~2/4,
50-384 Wroc{\l}aw, Poland}
\email{mlotkow@math.uni.wroc.pl}
\email{annaromanowicz85@gmail.com}
\begin{document}


\begin{abstract}
For delta operator $aD-bD^{p+1}$
we find the corresponding
polynomial sequence of binomial type
and relations with Fuss numbers.
In the case $D-\frac{1}{2}D^2$ we show that the corresponding
Bessel-Carlitz polynomials are moments of the convolution semigroup
of inverse Gaussian distributions.
We also find probability distributions $\nu_{t}$, $t>0$, for which
$\left\{y_{n}(t)\right\}$, the Bessel polynomials at $t$, is the moment sequence.
\end{abstract}

\maketitle

\section{Introduction}

A sequence $\left\{w_{n}(t)\right\}_{n=0}^{\infty}$ of polynomials is said to be of \textit{binomial type}
(see \cite{roman}) if $\deg w_{n}(t)=n$ and for every $n\ge0$ and $s,t\in\mathbb{R}$ we have
\begin{equation}\label{aintrbinomial}
w_{n}(s+t)=\sum_{k=0}^{n}\binom{n}{k}w_{k}(s)w_{n-k}(t).
\end{equation}

A linear operator $Q$ of the form
\begin{equation}\label{aintrdelta}
Q=\frac{c_1}{1!} D+\frac{c_2}{2!} D^2+\frac{c_3}{3!} D^3+\ldots,
\end{equation}
acting on the linear space $\mathbb{R}[x]$ of polynomials,
is called a \textit{delta operator} if $c_1\ne 0$.
Here $D$ denotes the derivative operator: $D1:=0$ and $D t^n:=n\cdot t^{n-1}$
for $n\ge1$. We will write $Q=g(D)$, where
\begin{equation}\label{aintrg}
g(x)=\frac{c_1}{1!}x+\frac{c_2}{2!} x^2+\frac{c_3}{3!} x^3+\ldots.
\end{equation}

There is one-to-one correspondence between sequences of binomial type
and delta operators, namely if $Q$ is a delta operator then there
is unique sequence $\left\{w_{n}(t)\right\}_{n=0}^{\infty}$ of binomial type
satisfying $Qw_{0}(t)=0$ and $Qw_{n}(t)=n\cdot w_{n-1}(t)$ for $n\ge1$.
These $w_{n}$ are called \textit{basic polynomials} for $Q$.

A natural way of obtaining a sequence of binomial type is to start with
a function $f$
which is analytic in a neighborhood of $0$:
\begin{equation}\label{aintrf}
f(x)=\frac{a_1}{1!}x+\frac{a_2}{2!} x^2+\frac{a_3}{3!} x^3+\ldots,
\end{equation}
with $a_1\ne0$.
Then a binomial sequence appears
in the Taylor expansion of $\exp\left(t\cdot f(x)\right)$, namely
\begin{equation}\label{aintrexp}
\exp\left(t\cdot f(x)\right)=\sum_{n=0}^{\infty}\frac{w_{n}(t)}{n!}x^n
\end{equation}
and the coefficients of $w_n$ are partial Bell polynomials of $a_{1},a_2,\ldots$.
More generally, we can merely assume that $f$ (as well as $g$) is a formal power series.
Then $f$ is the composition inverse of $g$: $f(g(x))=g(f(x))=x$.

The aim of the paper is to describe the sequences 
of polynomials of binomial type, which correspond to delta operators of the form
$Q=a D-b D^{p+1}$.
Then we discuss the special case, 
$D-\frac{1}{2}D^2$, studied by Carlitz \cite{carlitz}.
We find the corresponding semigroup of probability distributions
an also the family of distributions corresponding to the Bessel polynomials.

\section{The result}

\begin{theorem}
For $a\ne0$, $b\in\mathbb{R}$ and for $p\ge1$ let $Q:=a D-b D^{p+1}$.
Then the basic polynomials are as follows:
$w_0(t)=1$ and for $n\ge1$
\begin{equation}\label{bwn}
w_n(t)=\sum_{j=0}^{\left[\frac{n-1}{p}\right]}\frac{(n+j-1)!b^{j}}{j!\left(n-jp-1\right)!a^{n+j}}t^{n-jp}.
\end{equation}
In particular $w_1(t)=t/a$.
\end{theorem}

\begin{proof}
We have to show that if $n\ge1$ then $Q w_{n}(t)=n\cdot w_{n-1}(t)$.
It is obvious for $n=1$, so assume that $n\ge2$.
Then we have
\[
a Dw_n(t)=\sum_{j=0}^{\left[\frac{n-1}{p}\right]}\frac{(n+j-1)!(n-jp)b^j}{j!\left(n-jp-1\right)!a^{n-1+j}}t^{n-1-jp}
\]
\[
=n \left(\frac{t}{a}\right)^{n-1}
+\sum_{j=1}^{\left[\frac{n-1}{p}\right]}\frac{(n+j-1)!(n-jp)b^j}{j!\left(n-jp-1\right)!a^{n-1+j}}t^{n-1-jp}
\]
and
\[
b D^{p+1}w_n(t)=\sum_{j=0}^{\left[\frac{n-p-1}{p}\right]}
\frac{(n+j-1)!(n-jp)b^{j+1}}{j!\left(n-jp-p-1\right)!a^{n+j}}t^{n-1-jp-p}.
\]
Now we substitute $j':=j+1$:
\[
b D^{p+1}w_n(t)=\sum_{j'=1}^{\left[\frac{n-1}{p}\right]}
\frac{(n+j'-2)!(n-j'p+p)b^{j'}}{(j'-1)!\left(n-j'p-1\right)!a^{n-1+j'}}t^{n-1-j'p}.
\]
If $j\ge1$ and $jp+1<n$ then
\[
\frac{(n+j-1)!(n-jp)}{j!(n-jp-1)!}-\frac{(n+j-2)!(n-jp+p)}{(j-1)!(n-jp-1)!}
\]
\[
=\frac{(n+j-2)!\big[(n+j-1)(n-jp)-j(n-jp+p)\big]}{j!(n-jp-1)!}
\]
\[
=\frac{n(n+j-2)!(n-jp-1)}{j!(n-jp-1)!}=\frac{n(n+j-2)!}{j!(n-jp-2)!}
\]
and if $
jp+1=n$ then this difference is $0$.
Therefore we have
\[
\left(aD-bD^{p+1}\right)w_{n}(t)=n \left(\frac{t}{a}\right)^{n-1}
+\sum_{j=1}^{\left[\frac{n-2}{p}\right]}\frac{n(n+j-2)!b^{j}}{j!\left(n-jp-2\right)!a^{n-1+j}}t^{n-1-jp}
\]
\[
=n\sum_{j=0}^{\left[\frac{n-2}{p}\right]}\frac{(n+j-2)!b^{j}}{j!\left(n-jp-2\right)!a^{n-1+j}}t^{n-1-jp}
=n\cdot w_{n-1}(t),
\]
which concludes the proof.
\end{proof}

Recall, that \textit{Fuss numbers of order $p$} are given
by
$\binom{np+1}{n}\frac{1}{np+1}$
and the corresponding generating function
\begin{equation}\label{bgeneratingp}
\mathcal{B}_{p}(x):=\sum_{n=0}^{\infty}\binom{np+1}{n}\frac{x^n}{np+1}
\end{equation}
is determined by the equation
\begin{equation}\label{bfussfunction}
\mathcal{B}_{p}(x)=1+x\cdot \mathcal{B}_{p}(x)^p.
\end{equation}
In particular
\begin{equation}\label{bcatalanfunction}
\mathcal{B}_{2}(x)=\frac{1-\sqrt{1-4x}}{2x}.
\end{equation}
For more details, as well as combinatorial applications, we refer to \cite{gkp}.

Now we can exhibit the function $f$ corresponding
to the operator $aD-bD^{p+1}$ and to the polynomials (\ref{bwn}).

\begin{corollary}\label{bcorollary}
For the polynomials (\ref{bwn}) we have
\[
\exp\left(t\cdot f(x)\right)=\sum_{n=0}^{\infty}\frac{w_{n}(t)}{n!}x^n,
\]
where
\begin{equation}\label{bformulacor}
f(x)=\frac{x}{a}\cdot\mathcal{B}_{p+1}\left(\frac{bx^p}{a^{p+1}}\right).
\end{equation}
\end{corollary}

\begin{proof}
Since $f$ is the inverse function for $g(x):=ax-b x^{p+1}$,
it satisfies $af(x)=x+bf(x)^{p+1}$, which is equivalent  to
\[
\left(\frac{a f(x)}{x}\right)=1+\frac{bx^p}{a^{p+1}}\left(\frac{a f(x)}{x}\right)^{p+1}.
\]
Comparing with (\ref{bfussfunction}), we see that $af(x)/x=\mathcal{B}_{p+1}\left(bx^p/a^{p+1}\right)$.

Alternatively, we could apply the formula $a_n=w_{n}'(0)$ for the coefficients in (\ref{aintrf}).
\end{proof}

\section{The operator $D-\frac{1}{2}D^2$}

One important source of sequences of binomial type are moments
of convolution semigroups of probability measures on the real line.
In this part we describe an example of such semigroup,
which corresponds to the delta operator $D-\frac{1}{2}D^2$.
For more details we refer to \cite{carlitz,grosswald,krallfrink,roman}
and to the entry $\mathrm{A001497}$ in~\cite{oeis}.

In view of (\ref{bwn}), the sequence of polynomials
of binomial type corresponding to the delta operator $D-\frac{1}{2}D^2$ is
\begin{equation}\label{cbinomialpoly}
w_n(t)=\sum_{j=0}^{n-1}\frac{(n+j-1)!\, t^{n-j}}{j!\left(n-j-1\right)!\,2^{j}}
=\sum_{k=1}^{n}\frac{(2n-k-1)!\,t^{k}}{(n-k)!\left(k-1\right)!\,2^{n-k}},
\end{equation}
with $w_{0}(t)=1$. They are related to the \textit{Bessel polynomials}
\begin{equation}\label{cbesselpolynomials}
y_{n}(t)=\sum_{j=0}^{n}\frac{(n+j)!}{j!\left(n-j\right)!}\left(\frac{t}{2}\right)^{j}
=e^{1/t}\sqrt{\frac{2}{\pi t}}K_{-n-1/2}\left(1/t\right),
\end{equation}
where $K_{\nu}(z)$ denotes the modified Bessel function of the second kind.
Namely, for $n\ge1$ we have
\begin{equation}\label{cwnbesselfunction}
w_{n}(t)=t^n y_{n-1}\left(1/t\right)
=t^n e^{t}\sqrt{\frac{2t}{\pi}} K_{1/2-n}(t)
\end{equation}
for $n\ge0$.

Applying Corollary~\ref{bcorollary} and (\ref{bcatalanfunction}) we get
\begin{equation}
f(x)=x\cdot \mathcal{B}_{2}(x/2)=1-\sqrt{1-2x}=\sum_{n=1}^{\infty}\frac{a_n}{n!}x^n,
\end{equation}
where $a_n=(2n-3)!!\,$.
The function $f(x\mathrm{i})$ admits
\textit{Kolmogorov representation} (see formula (7.15) in \cite{steutel}) as 
\begin{equation}
1-\sqrt{1-2x\mathrm{i}}=x\mathrm{i}
+\int_{0}^{+\infty}\big(e^{ux\mathrm{i}}-1-ux\mathrm{i}\big)\frac{e^{-u/2}}{\sqrt{2\pi u^3}}\,du,
\end{equation}
with the probability density function ${\sqrt{u}e^{-u/2}}/{\sqrt{2\pi}}$
on $[0,+\infty)$. Therefore
\[\phi(x)=\exp\left(1-\sqrt{1-2x\mathrm{i}}\right)\]
is the characteristic function of some infinitely
divisible probability measure.
It turns out that the corresponding convolution semigroup $\left\{\mu_t\right\}_{t>0}$
is contained in the family of
\textit{inverse Gaussian distributions}, see 24.3 in \cite{balakrishnannevzorov}.


\begin{theorem}\label{dthmcarlitzmeasure}
For $t>0$ define probability distribution $\mu_{t}:=\rho_{t}(u)\,du$, where
\begin{equation}\label{crhote}
\rho_{t}(u):=\frac{t\cdot \exp\left(\frac{-(u-t)^2}{2u}\right)}{\sqrt{2\pi u^3}}
\end{equation}
for $u>0$ and $\rho_{t}(u)=0$ for $u\le0$.
Then $\left\{\mu_t\right\}_{t>0}$ is a convolution semigroup,
$\exp\left(t-t\sqrt{1-2x\mathrm{i}}\right)$ is the characteristic function of $\mu_t$,
and $\left\{w_{n}(t)\right\}_{n=0}^{\infty}$, defined by (\ref{cbinomialpoly}),
is the moment sequence of $\mu_{t}$.
\end{theorem}

\begin{proof}
It is sufficient to check moments of $\mu_{t}$. Substituting $u:=2v$
and applying formula:
\begin{equation}\label{cbesselfunctionformula}
K_{p}(t)=\frac{1}{2}\left(\frac{t}{2}\right)^{p}
\int_{0}^{\infty}\exp\left(-v-\frac{t^2}{4v}\right)\frac{dv}{v^{p+1}}
\end{equation}
(see  (10.32.10) in \cite{olver}), we obtain
\[
\int_{0}^{\infty}u^n\frac{t\cdot \exp\left(\frac{-(u-t)^2}{2u}\right)}{\sqrt{2\pi u^3}}\,du
=\frac{t e^t}{\sqrt{2\pi}}\int_{0}^{\infty}\exp\left(\frac{-u}{2}-\frac{t^2}{2u}\right)u^{n-3/2}\,du
\]
\[
=\frac{t e^t 2^{n-1}}{\sqrt{\pi}}\int_{0}^{\infty}\exp\left(-v-\frac{t^2}{4v}\right)v^{n-3/2}\,dv
=\frac{t e^t 2^{n-1}}{\sqrt{\pi}} 2\left(\frac{t}{2}\right)^{n-1/2}K_{1/2-n}(t),
\]
which, by (\ref{cwnbesselfunction}), is equal to $w_n(t)$.
\end{proof}

\section{Probability measures corresponding to the Bessel polynomials}

In this part we are going to give some remarks concerning Bessel polynomials~(\ref{cbesselpolynomials}).
First we compute the exponential generating function (cf. formula (6.2) in \cite{grosswald}).

\begin{proposition}\label{dbesselmoments}
For the exponential generating function of the sequence $\left\{y_{n}(t)\right\}$ we have
\begin{equation}\label{dgeneratingbessel}
\sum_{n=0}^{\infty}\frac{y_{n}(t)}{n!}x^n=\frac{\exp\left(\frac{1}{t}-\frac{1}{t}\sqrt{1-2tx}\right)}{\sqrt{1-2tx}}.
\end{equation}
\end{proposition}

\begin{proof}
By (\ref{cwnbesselfunction}) we have
\[
\sum_{n=0}^{\infty}\frac{y_{n}(t)}{n!}x^n
=\sum_{n=0}^{\infty}\frac{t^{n+1} w_{n+1}(1/t)}{n!}x^n
=\sum_{n=1}^{\infty}\frac{t^{n} w_{n}(1/t)}{(n-1)!}x^{n-1}\]
\[
=\frac{d}{dx}\left(\sum_{n=0}^{\infty}\frac{t^{n} w_{n}(1/t)}{n!}x^n\right)
=\frac{d}{dx}\exp\left(\frac{1}{t}-\frac{1}{t}\sqrt{1-2tx}\right),
\]
which leads to~(\ref{dgeneratingbessel}).
\end{proof}

Now we represent the Bessel polynomials (\ref{cbesselpolynomials}) as a moment sequence.

\begin{theorem}
For $n\ge0$  and $t>0$ we have
\begin{equation}
y_{n}(t)=\int_{0}^{\infty} u^n
\frac{\exp\left(\frac{-(u-1)^2}{2tu}\right)}{\sqrt{2\pi t u}}\,du.
\end{equation}
\end{theorem}

\begin{proof}
Substituting $u:=2tv$, applying (\ref{cbesselfunctionformula}) and (\ref{cbesselpolynomials}) we get:
\[
\int_{0}^{\infty} u^n
\frac{\exp\left(\frac{-(u-1)^2}{2tu}\right)}{\sqrt{2\pi t u}}\,du
=\frac{e^{1/t}}{\sqrt{2\pi t}}\int_{0}^{\infty} u^{n-1/2}
\exp\left(\frac{-u}{2t}-\frac{1}{2tu}\right)\,du.
\]
\[
=\frac{e^{1/t}(2t)^{n+1/2}}{\sqrt{2\pi t}}\int_{0}^{\infty} v^{n-1/2}
\exp\left(-v-\frac{1}{4t^2 v}\right)\,dv
\]
\[
=e^{1/t}\sqrt{\frac{2}{\pi t}}K_{-n-1/2}\left(1/t\right)=y_{n}(t),
\]
which concludes the proof.

Alternatively, we could apply Theorem~\ref{dthmcarlitzmeasure} and relation~(\ref{cwnbesselfunction}).
\end{proof}

Denote by $\nu_{t}$ the corresponding probability measure, i.e.
\begin{equation}
\nu_{t}:=\frac{\exp\left(\frac{-(u-1)^2}{2tu}\right)}{\sqrt{2\pi t u}}\chi_{(0,+\infty)}(u)\,du.
\end{equation}
Although $\left\{\nu_t\right\}_{t>0}$ is not a convolution semigroup, we will see that every $\nu_{t}$ is infinitely divisible.

\begin{theorem}
For $t>0$ we have
\begin{equation}\label{dconvolution}
\nu_{t}=\gamma_t*\mathbf{D}_{t}\mu_{1/t},
\end{equation}
where $\gamma_{t}$ denotes the gamma distribution with shape $1/2$ and scale $2t$:
\begin{equation}
\gamma_{t}=\frac{\exp\left(\frac{-u}{2t}\right)}{\sqrt{2\pi t u}}\chi_{(0,+\infty)}(u)\,du
\end{equation}
and $\mathbf{D}_{t}\mu_{1/t}$ is the dilation of $\mu_{1/t}$ by $t$:
\begin{equation}
\mathbf{D}_{t}\mu_{1/t}=
\frac{\exp\left(\frac{-(u-1)^2}{2tu}\right)}{\sqrt{2\pi t u^3}}\chi_{(0,+\infty)}(u)\,du.
\end{equation}
In particular, $\nu_{t}$ is infinitely divisible.
\end{theorem}

\begin{proof}
From Theorem~\ref{dbesselmoments} we see that the characteristic function of $\nu_{t}$:
\begin{equation}
\psi_{t}(x)=
\frac{\exp\left(\frac{1}{t}-\frac{1}{t}\sqrt{1-2tx\mathrm{i}}\right)}{\sqrt{1-2tx\mathrm{i}}},
\end{equation}
is the product of
\[
\frac{1}{\sqrt{1-2tx\mathrm{i}}},
\]
the characteristic function of $\gamma_{t}$ and
\[
\exp\left(\frac{1}{t}-\frac{1}{t}\sqrt{1-2tx\mathrm{i}}\right),
\]
the characteristic function of $\mathbf{D}_{t}\mu_{1/t}$, which proves (\ref{dconvolution}).
Since both $\gamma_t$ and $\mathbf{D}_{t}\mu_{1/t}$ are infinitely
divisible, so is their convolution $\nu_{t}$.
\end{proof}

Let us list some interesting integer sequences which arise from
the polynomials (\ref{cbinomialpoly}) and (\ref{cbesselpolynomials}),
together with their numbers in the On-line Encyclopedia
of Integer Sequences \cite{oeis} and the corresponding probability distribution.
For their combinatorial applications we refer to~\cite{oeis}:
\begin{enumerate}
\item $\mathrm{A144301}$: $w_{n}(1)$, moments of $\mu_1$,
\item $\mathrm{A107104}$: $w_{n}(2)$, moments of $\mu_{2}$,
\item $\mathrm{A043301}$: $w_{n+1}(2)/2$, moments of the density function $u\cdot\rho_{2}(u)/2$,
\item $\mathrm{A080893}$: $2^n\cdot w_{n}(1/2)$, moments of the density function $\rho_{1/2}\left(u/2\right)/2$,
\item $\mathrm{A001515}$: $y_{n}(1)$, moments of $\nu_{1}$,
\item $\mathrm{A001517}$: $y_{n}(2)$, moments of $\nu_{2}$,
\item $\mathrm{A001518}$: $y_{n}(3)$, moments of $\nu_{3}$,
\item $\mathrm{A065919}$: $y_{n}(4)$, moments of $\nu_{4}$.
\end{enumerate}

\end{document}